\documentclass{amsart}
\usepackage[all,knot]{xy}
\xyoption{arc} 
\usepackage{amsfonts}
\usepackage{array}
\usepackage{euscript}
\usepackage{epsfig, psfrag, epic}
\usepackage{graphics}
\usepackage{xy}
\usepackage{multirow}
\usepackage{amsmath}
\usepackage{subfigure}
\usepackage{verbatim}

\title{Unstable Vassiliev Theory} 

\author{Chad Giusti}

\newtheorem{thm}{Theorem}[section]

\newtheorem{lem}[thm]{Lemma}
\newtheorem{prop}[thm]{Proposition}

\theoremstyle{definition}
\newtheorem{defn}[thm]{Definition}

\newtheorem{example}[thm]{Example}

\newcommand{\into}{\hookrightarrow}
\newcommand{\dirlim}{\underrightarrow\lim}
\newcommand{\invlim}{\underleftarrow\lim}

\newcommand{\e}{{\mathbf e}}
\newcommand{\te}{\tilde{\e}}
\newcommand{\f}{{\mathbf f}}
\newcommand{\tf}{\tilde{\f}}
\newcommand{\tS}{\tilde{S}}
\newcommand{\vsig}{\vec{\sigma}}

\newcommand{\bC}{\mathbf{C}}
\newcommand{\Cell}{\text{\sc Cell}}

\newcommand{\cx}{\textrm{\sc cx}}

\newcommand{\refT}[1]{Theorem~\ref{T:#1}}

\newcommand{\refD}[1]{Definition~\ref{D:#1}}
\newcommand{\refL}[1]{Lemma~\ref{L:#1}}

\newcommand{\refF}[1]{Figure~\ref{F:#1}}

\begin{document}

\begin{abstract}    
We construct an inverse system of unstable Vassiliev spectral sequences on the spaces of plumbers' knots. Utilizing the cell structure on these spaces, we extend the notion of Vassiliev derivative to all singularity types of plumbers' knots. 
\end{abstract}

\maketitle

In \cite{Vas90}, Vassiliev initiated the study of finite-type invariants by constructing the spectral sequence which bears his name and analyzing the combinatorics of its $E_1$-page. Due to the highly technical nature of its construction, including the use of the weak transversality theorem to perturb polynomial mapping spaces, few other than Vassiliev himself have built upon this approach. The principal tool in the study of finite-type invariants has instead has been the notion of the Vassiliev derivative introduced by Birman and Lin \cite{BirmanLin} and made popular by Bar-Natan \cite{NatanOnVas}. 

We believe that there remains a great deal to be learned through a geometric analysis of the discriminant. Several authors, notably Randell \cite{Randell98, Randell02}, Calvo \cite{Calvo} and Stanford (in unpublished work), have approached this problem by replacing Vassiliev's choice of polynomial knot spaces by the spaces of (piecewise-linear) stick knots. The discriminant in these spaces is constructed from partial cubic hypersurfaces and has hardly been more amenable to comprehensive description. 

We base our construction on a directed system of spaces called ``plumbers knots", constructed by the author as a model for classical knot theory in \cite{Plumbers}. The discriminant in this setting is a union of partial hyperplanes and admits a natural cell structure. In this context, we extend the notion of Vassiliev derivative to work for any singularity of plumbers' maps and introduce the notion of a Vassiliev system for a knot invariant. Using this structure, we produce an honest inverse system of ``unstable'' Vassiliev spectral sequences whose limiting sequence's $E_\infty$ page contains that of the classical Vassiliev spectral sequence. In contrast to the Vassiliev's ``stable range'' construction, each such unstable sequence carries information about all singularities arising in the space of plumbers' curves on which it is constructed. In exchange for more intricate combinatorics, this provides us with complete data regarding the evolution of knot invariants through the system.

\subsection*{Acknowledgements} The author would like to thank his advisor, Dev Sinha, for his support, expertise and patience during the development of this material.

\tableofcontents

\section{Conceptual Vassiliev theory}\label{conceptual}

The foundation of Vassiliev's approach to knot theory in \cite{Vas90} is that rather than analyzing properties of individual knots he applies the tools of algebraic topology to the space of all knots, $\mathcal{K}$. To ease digestion of the details of our construction of the ``unstable" Vassiliev spectral sequence, we begin with an exposition of the conceptual framework Vassiliev used in his original spectral sequence construction.

The complement of $\mathcal{K}$ in the space of all immersions, or \emph{discriminant} of $\mathcal{K}$, intersects itself in arbitrarily complex ways, as conceptually illustrated in \refF{discriminant}. Vassiliev's initial object was to ``resolve" this singular space, replacing it with a union of smooth objects. The most natural such construction replaces points of the discriminant corresponding to curves with $n$ transverse double points by $n-1$ simplices, as in \refF{resolved}. More complex singularities do not fit easily into this picture, but do not ``generically" occur, so Vassiliev discarded them in favor of approchable combinatorics. Correspondingly, the spectral sequence he constructs fails to see any data carried in these singularities. Thus, for example, the completeness question for finite-type invariants is one of whether the remaining information is ``dense" in the collection of all knot invariants and not merely a $\lim^1$ question as is commonly assumed.

Vassiliev then introduced a logical ordering on these singularities by ``complexity", providing a filtration on the resolved discriminant. In the filtration quotient, the boundaries of the simplices introduced in the resolution collapse, leaving a collection of combinatorial codimension one cycles which live on the $E^1$ page of the spectral sequence of the filtration.

From the standpoint of algebraic topology, knot invariants are classes in $H^0(\mathcal{K})$. Recall that the Alexander dual to a zero dimensional reduced cocycle $[\alpha]$ in a subspace $X\subseteq \mathbb{R}^n$ is a codimension one cycle $[\alpha^\vee]$ in $(\mathbb{R}^n\setminus{X})^+$. These cycles have as canonical chain representatives the sum of the codimension one chains of $X$ with coefficients given by the difference in values of $[\alpha]$ on its cobounding regions. Intuitively, the Alexander dual of a zero cocycle is the collection of its ``derivatives" as one changes components along a path like that in \refF{discriminant}.

What Vassiliev discovered is that the cycles on the $E^1$ page of his spectral sequence correspond to a notion of higher derivatives. At the chain level, the coefficient of each simplex encodes the change of value of an invariant of curves with fewer singular points. For example, a path between isotopy classes of curves with a single transverse double point generically passes through a finite number of regions corresponding to curves with two double points, potentially changing the coefficient of the Alexander dual at each crossing, suggesting a ``second derivative".  Those linear combinations of derivatives which survive to the $E^{\infty}$ page correspond to knot invariants.

Each cocycle in the $\mathcal{E}^1_{-n,n}$ line of Vassiliev's spectral sequence is associated to a weight system, so cocycles on the $\mathcal{E}^\infty$ page are represented by linear combinations of such. Denote by $\mathcal{FT}_\bullet = \bigoplus_n \mathcal{E}_{-n,n}^\infty$, the associated graded to this total degree zero line at the infinity page, and note that (with field coefficients) $\mathcal{FT}_\bullet$ is isomorphic to a subset of $\bar{H}^0(\mathcal{K})$. Call a non-zero element of $\bar{H}^0(\mathcal{K})$ which arises from $\mathcal{FT}_n$ an invariant of type $n$, and the collection of all such \emph{finite type} or \emph{Vassiliev} invariants.

\begin{figure}
\begin{center}
\subfigure[A schematic diagram for a section of the discriminant of $\mathcal{K}$]{
\psfrag{A}{$A$}
\psfrag{B}{$B$}
\psfrag{C}{(b)}
\psfrag{D}{(c)}
\psfrag{E}{(d)}
\includegraphics[width=10cm]{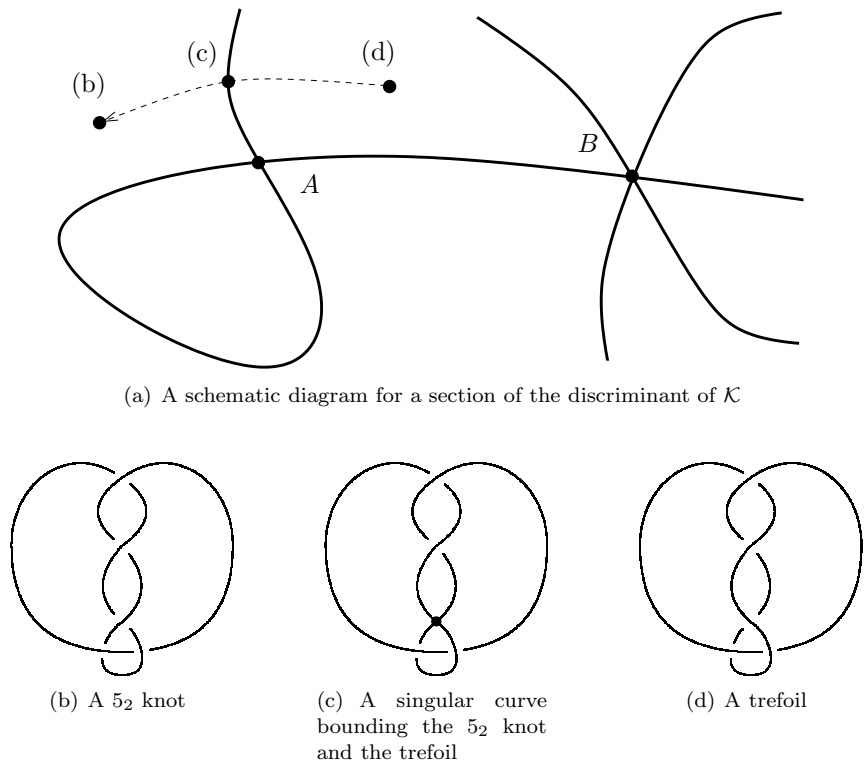}
}

\bigskip
\subfigure[A $5_2$ knot]{
\xy
(0,0)*{}="A";
(-2.56,3.2)*{\hole}="B";
(-2.56,3.2)*{}="B'";
(0,7.04)*{\hole}="C";
(0,7.04)*{}="C'";
(0,17.2)*{\hole}="D";
(0,17.2)*{}="D'";
(0,26.24)*{\hole}="E";
(0,26.24)*{}="E'";
(14.72,17.28)*{}="F";
(2.56,3.2)*{\hole}="G";
(2.56,3.2)*{}="G'";
(-14.72,17.28)*{}="H";
"A";"B" **\crv{(-3.2,0)} ;
"B";"C'" **\crv{(-1.92,5.76)} ;
"C'";"D" **\crv{(5.12,11.52)} ;
"D";"E'" **\crv{(-6.4,21.76)};
"E'";"F" **\crv{(7.04,31.36) & (14.72,25.6)};
"F";"G" **\crv{(14.72,4.48)};
"G";"B'" **\crv{(0,2.88)};
"B'";"H" **\crv{(-14.72,4.48)};
"H";"E" **\crv{(-14.72,25.6) & (-7.04,31.36)};
"E";"D'" **\crv{(6.4,21.76)};
"D'";"C" **\crv{(-5.12,11.52)};
"C";"G'" **\crv{(1.92,5.76)};
"G'";"A" **\crv{(3.2,0)};
\endxy 
}\hspace{1cm}
\subfigure[A singular curve bounding the $5_2$ knot and the trefoil]{
\xy
(0,0)*{}="A";
(-2.56,3.2)*{\hole}="B";
(-2.56,3.2)*{}="B'";
(0,7.04)*{\hole}="C";
(0,7.04)*{}="C'";
(0,7.04)*{\bullet}="C''";
(0,17.2)*{\hole}="D";
(0,17.2)*{}="D'";
(0,26.24)*{\hole}="E";
(0,26.24)*{}="E'";
(14.72,17.28)*{}="F";
(2.56,3.2)*{\hole}="G";
(2.56,3.2)*{}="G'";
(-14.72,17.28)*{}="H";
"A";"B" **\crv{(-3.2,0)} ;
"B";"C'" **\crv{(-1.92,5.76)} ;
"C'";"D" **\crv{(5.12,11.52)} ;
"D";"E'" **\crv{(-6.4,21.76)};
"E'";"F" **\crv{(7.04,31.36) & (14.72,25.6)};
"F";"G" **\crv{(14.72,4.48)};
"G";"B'" **\crv{(0,2.88)};
"B'";"H" **\crv{(-14.72,4.48)};
"H";"E" **\crv{(-14.72,25.6) & (-7.04,31.36)};
"E";"D'" **\crv{(6.4,21.76)};
"D'";"C'" **\crv{(-5.12,11.52)};
"C'";"G'" **\crv{(1.92,5.76)};
"G'";"A" **\crv{(3.2,0)};
\endxy 
}\hspace{1cm}
\subfigure[A trefoil]{
\xy
(0,0)*{}="A";
(-2.56,3.2)*{\hole}="B";
(-2.56,3.2)*{}="B'";
(0,7.04)*{\hole}="C";
(0,7.04)*{}="C'";
(0,17.2)*{\hole}="D";
(0,17.2)*{}="D'";
(0,26.24)*{\hole}="E";
(0,26.24)*{}="E'";
(14.72,17.28)*{}="F";
(2.56,3.2)*{\hole}="G";
(2.56,3.2)*{}="G'";
(-14.72,17.28)*{}="H";
"A";"B" **\crv{(-3.2,0)} ;
"B";"C" **\crv{(-1.92,5.76)} ;
"C";"D" **\crv{(5.12,11.52)} ;
"D";"E'" **\crv{(-6.4,21.76)};
"E'";"F" **\crv{(7.04,31.36) & (14.72,25.6)};
"F";"G" **\crv{(14.72,4.48)};
"G";"B'" **\crv{(0,2.88)};
"B'";"H" **\crv{(-14.72,4.48)};
"H";"E" **\crv{(-14.72,25.6) & (-7.04,31.36)};
"E";"D'" **\crv{(6.4,21.76)};
"D'";"C'" **\crv{(-5.12,11.52)};
"C'";"G'" **\crv{(1.92,5.76)};
"G'";"A" **\crv{(3.2,0)};
\endxy 

}
\end{center}
\caption{A cross-section of the discriminant of the knot space.}
\label{F:discriminant}
\end{figure}

\begin{figure}
\psfrag{a}{$\alpha$}
\psfrag{b}{$\beta$}
\vspace{30pt}
\begin{center}
\includegraphics[width=10cm]{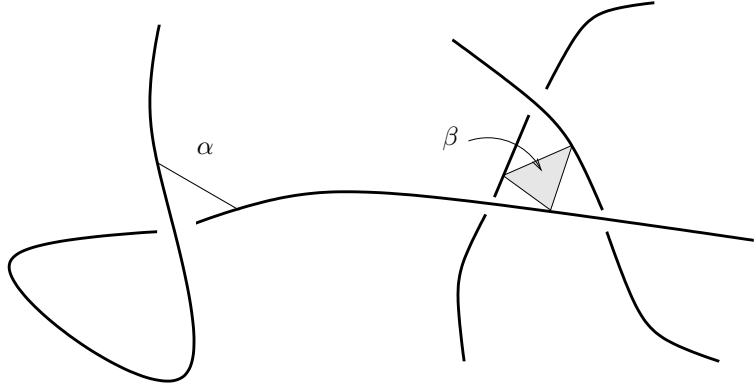}
\end{center}
\caption{Vassiliev's resolution of the discriminant; under the canonical projection, the simplex $\alpha$ maps to the point $A$ in \refF{discriminant} and $\beta$ to $B$.}
\label{F:resolved}
\end{figure}

In this paper, we revisit Vassiliev's construction in a more geometric context, the spaces of plumbers' curves. The restrictive nature of the singularity types which can occur in plumbers' curves lets us work on their entire discriminant, allowing us to define the Vassiliev derivative for all singularity types of plumbers' curves. Closely related to this, we are able to construct spectral sequences which retain all singularity information. 

\section{Plumbers' curves}

In \cite{Plumbers}, the author develops a finite-complexity knot theory called \emph{plumbers' knots}, the pertinent details of which we now briefly recall.

An \emph{$m$-move plumbers' curve} $\phi_\mathbf{v}\colon\thinspace I \to [0,1]^3$ is uniquely determined by a collection $\mathbf{v}$ of $(m-1)$ vertices in $(0,1)^3$. The image of such a map moves along segments parallel to the coordinate axes $x$, $y$ and $z$ (in this order) from vertex to vertex, starting at the origin and ending at $(1,1,1)$. See \refF{plumbersknot} for an example.

The space $P_m$ of all such is homeomorphic to $\left((0,1)^3\right)^{m-1}$. Each segment parallel to an axis is called a \emph{pipe}, and pipes which are separated by three or more intervening pipes are called \emph{distant}. $K_m\subseteq P_m$ is the space of \emph{m-move plumbers' knots}, consisting of those plumbers' maps for which distant pipes do not intersect. One feature that distinguishes this theory from that of PL knots is that our definition of knot allows for up to two adjacent zero-length pipes in a knot. Additionally, there are stabilization maps $\iota_m\colon\thinspace P_m \into P_{m+1}$, under which $\dirlim K_m$ has the weak homotopy type of the space of long knots, so this system is a model for classical knot theory.

\begin{figure}
\begin{center}
\includegraphics[width=7cm]{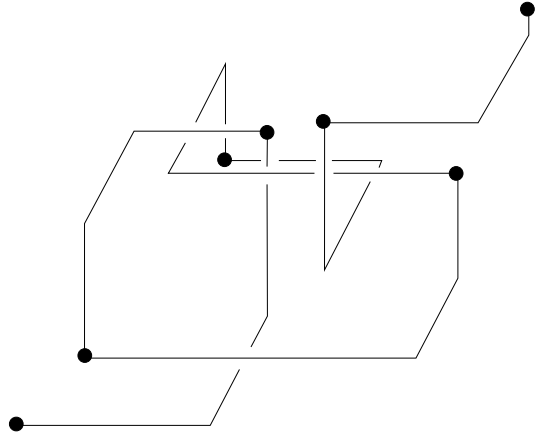}
\end{center}
\caption{A plumbers' knot of 6 moves.}
\label{F:plumbersknot}
\end{figure}

The space of plumbers' maps of $m$ moves admits a cellular decomposition generated by open cells homeomorphic to $\left(\Delta^{m-1}\right)^{\times 3}$. Each such cell is indexed by a triple of permutations of $(m-1)$ elements which describe the order in which the vertices appear when projected onto the $x$, $y$ or $z$-axes respectively. Write $\e(\sigma_x, \sigma_y, \sigma_z)$, $\sigma_x, \sigma_y, \sigma_z \in \Sigma_{m-1}$, for such a cell and $\Cell_\bullet(P_m)$ for the (polyhedral) CW structure defined by the cells. Where possible, we will abbreviate the triple $(\sigma_x, \sigma_y, \sigma_z)$ as $\vsig$ and write, for example, $\rho_x\vsig = (\rho_x\sigma_x, \sigma_y, \sigma_z)$ for the left action of $\Sigma_{m-1}$ on $\Sigma_{m-1}\times \{x,y,z\}$ in the indicated coordinate.

Boundaries of cells $\e(\vsig)\in\Cell_{3m-4}(P_m)$ are indexed by collections of coordinate equalities on the vertices which define their elements. We encode such an equality as a transposition decorated with a label indicating which coordinate it involves. For example, $(1\;2)_x$ means that the first and second vertex share x-coordinates, which is a valid boundary condition precisely when $(1\;2) = (\sigma_x(i)\;\sigma_x(i+1))$ for some choice of $i$.

Given a collection $\tau$ of transpositions and a cell $\e\in \Cell_\bullet(P_m)$ for whose elements all of the equalities indexed by $\tau$ hold, we say the cell \emph{respects $\tau$}. For example, $\e(3142_x, 4132_y, 1324_z)$ has several boundary cells which respect the set $\tau = \{(1\, 3)_x, (2\, 4)_x, (1\, 4)_y, (1\, 3)_y\}$, all of which also respect the set $\tau'=\{(1\, 3)_x, (1\, 4)_y\}$.


We will now establish notation for these boundary cells which will allow us to more easily describe the geometry of the spaces of singular plumbers curves. Vassiliev chooses to consider only collections of $k$-fold transverse intersections and points with vanishing derivative. We will instead investigate all possible singularities of plumbers' curves. Although the possible singularity types are more restrictive than those appearing for all smooth curves, our decision to retain all singularity information will result in more complex combinatorics.

\begin{defn}
Let $[\mathbf{m}]=\{1,\dots,m\}$ and $\mathcal{P}(S)$ be the power set of a set $S$. Fix a triple of permutations $\vsig \in \Sigma_{m-1}\times\{x,y,z\}$. Given an ordered pair $(i, d)\in \mathbf{[m-1]} \times \{x,y,z\}$, call $i$ the \emph{index} and $d$ the \emph{direction}.

We say that a set $C \in \mathcal{P}(\mathbf{[m-1]} \times \{x,y,z\})$ is \emph{admissible for $\vsig$} if all of the elements of $C$ share the same direction $\alpha$ and its indices are of the form $\sigma_\alpha(\{i, i+1,\dots, i+k\})$ for some $i,k$.
\end{defn}

For example, $C=\{1, 2, 4\}_x$ is admissible for $(3142_x, 4132_y, 1324_z)$. Sets which are admissible for $\vsig$ index the collections of coordinate equalities which can occur in the boundary of the cell $\e(\vsig)$.

Given a set $C$ which is admissible for $\vsig$, we can produce a collection of transpositions $\tau(C)$ which describe the coordinate equalities in $C$ compatibly with the order of the vertices in $\e(\vsig)$. To do so, we simply read off the transpositions in the order they appear in $\vsig$.

\begin{defn}
Define $\tau(C, \vsig) = \{(\sigma_\alpha(i)\;\sigma_\alpha(i+1)), (\sigma_\alpha(i+1)\;\sigma_\alpha(i+2)),\dots, (\sigma_\alpha(i+k-1)\;\sigma_\alpha(i+k))\})$.

We say such a collection of transpositions is \emph{sequential for $\vsig$}. When $\vsig$ is clear from context, we will supress it from notation.
\end{defn}

In the example above, $\tau(C, \vsig) = \{(1\;4)_x, (4\;2)_x\}$.

\begin{defn}\label{D:boundarycell}
Fix a triple of permutations $\vsig \in \Sigma_{m-1}\times\{x,y,z\}$. Let $\bC=\{C_1, C_2,\dots,C_k\}$ be a partition of $\mathbf{[m-1]} \times \{x,y,z\}$ into sets which are admissible for $\vsig$. Denote by $\e(\vsig; \bC)$ the cell of plumbers' curves obtained by setting equal precisely those coordinates of vertices which appear in the same $C_i$ and otherwise respecting the inequalities induced by $\vsig$. 
\end{defn}

Such a cell is a boundary of $\e(\vsig)$ of codimension $\sum |C_i| - |\bC|$. We will omit singletons when writing $\bC$, as these induce no equalities in the coordinates.

To continue our example with $\vsig = (3142_x, 4132_y, 1324_z)$, there is a boundary cell of $\e(\vsig)$ given by $\e(\vsig; \{1, 3\}_x, \{2, 4\}_x, \{1, 3, 4\}_y)$ whose codimension is $2 + 2 + 3 - 3 = 4$. Any boundary cell of $\e(\vsig)$ which respects $\tau = \{(1\, 3)_x, (2\, 4)_x, (1\, 4)_y, (1\, 3)_y\}$ is also a boundary of this cell.

We remark that cells of codimension one or greater are not uniquely named; we can rearrange any indices in $\vsig$ which appear in the same component of $\bC$ and to obtain another permutation $\vsig'$ and another label for the same cell, $\e(\vsig', \bC)$.  Another name for our cell is thus $\e(1342_x, 1342_y, 1324_z; \{1, 3\}_x, \{2, 4\}_x, \{1, 3, 4\}_y)$. This flexible naming convention will simplify the formula for the Vassiliev derivative.

\begin{defn}
Let $\vsig$ and $\bC$ be as in \refD{boundarycell}. Define $\Sigma_{\bC} = \prod_{i=1}^k \Sigma_{C_i}$, where $\Sigma_{C_i}$ is the symmetric group on the elements of $C_i$.
\end{defn}

All possible names for a given cell $\e(\vsig; \bC)$ are given by $\e(\rho\vsig;\bC)$ for $\rho\in \Sigma_\bC$. \medskip

The principal object of interest here, $S_m = P_m\setminus K_m$, is the \emph{discriminant}, consisting of all singular plumbers' maps. $S_m$ inherits a cellular structure from $P_m$ in the form of a closed $3m-4$ dimensional subcomplex $\Cell_\bullet(S_m)\subseteq \Cell_\bullet(P_m)$.

This cell structure leads to a convenient decomposition of the space $S_m$. Denote by $I \choose 2$ the collection of two element subsets of $I$. 

\begin{defn}
Define $\mathcal{S}_m$ be the $m$th \emph{coincidence category}, whose objects are non-empty elements of $\mathcal{P}\left({[\mathbf{m-1}] \choose 2}\times\{x,y,z\}\right)$ and whose morphisms are reverse inclusions.
\end{defn}

Elements of $\mathcal{S}_m$ are precisely our collections of transpositions, as in \refD{boundarycell}.

\begin{defn} \label{D:singcat}
Let $B_m \colon\thinspace \mathcal{S}_m \to \mathbf{Top}$ be the covariant functor given by $B_m(\tau) = \{\phi \in S_m : \phi \textrm{ respects } \tau\}$.
\end{defn}

Our analysis of the cell complex above now immediately gives us that

\begin{prop}
$S_m = \textrm{\upshape{colim}} B_m$.
\end{prop}


\section{Vassiliev theory in the plumbers' knot setting}

Using the spaces of plumbers' curves, we now construct our unstable version of the Vassiliev spectral sequence. The rigid geometry of this setting allows us to streamline the definition of the blowup of the discriminant and to explicitly introduce a formula for the Vassiliev derivative of any singularity type. 

\subsection{The homotopical blowup of the discriminant}

The problem of understanding the geometry of the discriminant is precisely that of understanding an arrangement of partial real hyperplanes. It is natural to encode this intersection data through simplices. The combinatorial description of the discriminant as a colimit gives us the information we require to perform this encoding using the homotopy colimit. 

\begin{defn}
The \emph{homotopical blowup of the discriminant} is $\tS_m = \textrm{hocolim} B_m$.
\end{defn}

The discriminant we describe is a Reedy fibrant space, so the following proposition is an instance of the general construction considered, for example, as Application 13.6 in Dugger's clear expository paper on homotopy colimits \cite{DuggerHocolim}. While blowing up the discriminant in this manner is a standard technique, Vassiliev's complexity filtration produces a spectral sequence which is not equivalent to the one  recorded in \cite{DuggerHocolim} which arises from the usual simplicial filtration.

\begin{prop}\label{P:blowupcorrect}
The projection map $\pi\colon\thinspace\tS_m \to S_m$ is a homotopy equivalence.
\end{prop}	

Using this definition, there is a straightforward cell structure $C_*(\tS_m)$ which lies over $\Cell_\bullet(S_m)$. Taking the homotopy colimit of $B_m$ results in each cell $\e(\vsig; \bC)\in \Cell_\bullet(S_m)$ being replaced by a product of that cell with a simplex whose vertices are labelled by the collection of all transpositions in $\Sigma_\bC$. This construction differs from Vassiliev's: rather than discovering a cell structure on the filtration quotients of the resolved discriminant, we lift the existing structure to the entire discriminant in a canonical fashion.

\begin{defn}
Let $\e=\e(\vsig; \bC)\in \Cell_\bullet(S_m)$ and let $\rho$ be a nonempty collection of transpositions in $\Sigma_\bC$. Denote by $\ast$ the topological join and by $\rho(C_i)$ the transpositions in $\rho$ with support on $C_i$.

Define $\te(\vsig; \bC; \rho) = \e \times \ast_{i=1}^\ell \Delta^{{\rho(C_i)\choose 2}-1}\in C_*(\tS_m)$ to be the face of the simplex ``lying over" $\e(\vsig;\bC)$ indexed by the elements of $\rho$. 
\end{defn}

By definition, the collection of all such cells $\te(\vsig; \bC; \rho)$ is a complete cell structure for $\tS_m$. See \refF{cellstruct} for an illustration of $C_*(\tS_m)$ in a simple case. Write $\pi_\#$ for the induced map $C_\ast(\tS_m) \to \Cell_\bullet(S_m)$ which ``forgets $\rho$". 

It will be useful to abuse notation and extend our naming conventions to the plumbers' knots, which by necessity have empty singularity data, denoting by $\te(\vsig; \bC;\emptyset)$ the cell $\e(\vsig) \in \Cell_{3m-3}(K_m)$.

\begin{figure}
\psfrag{e123-12}{\small $\te(\vsig; \{1_y, 2_y\}; \{(1\;2)_y\})$}
\psfrag{e132-13}{\small $\te((2\;5)_y\vsig; \{1_y, 5_y\}; \{(1\;5)_y\})$}
\psfrag{e123-23}{\small $\te(\vsig; \{2_y, 5_y\}; \{(2\;5)_y\})$}
\psfrag{e213-13}{\small $\te((1\;2)_y\vsig; \{1_y, 5_y\}; \{(1\;5)_y\})$}
\psfrag{e321-12}{\small $\te((1\;5)_y\vsig; \{1_y, 2_y\}; \{(1\;2)_y\})$}
\psfrag{e321-23}{\small$\te((1\;5)_y\vsig; \{2_y, 5_y\}; \{(2\;5)_y\})$}
\psfrag{e123-(12)}{\small$\te(\vsig; \{1_y, 2_y, 5_y\}; \{(1\;2)_y\})$}
\psfrag{e123-(23)}{\small$\te(\vsig; \{1_y, 2_y, 5_y\}; \{(2\;5)_y\})$}
\psfrag{e123-(13)}{\small$\te(\vsig; \{1_y, 2_y, 5_y\}; \{(1\;5)_y\})$}
\psfrag{e123-1213}{\small$\te(\vsig; \{1_y, 2_y, 5_y\}; \{(1\;2)_y, (1\;5)_y\})$}
\psfrag{e123-1223}{\small$\te(\vsig; \{1_y, 2_y, 5_y\}; \{(1\;2)_y, (2\;5)_y\})$}
\psfrag{e123-1323}{\small$\te(\vsig; \{1_y, 2_y, 5_y\}; \{(1\;5)_y, (2\;5)_y\})$}
\psfrag{e123-121323}{\small$\te(\vsig; \{1_y, 2_y, 5_y\};$}
\psfrag{e123-121323-2}{\small$\{(1\;2)_y, (1\;5)_y, (2\;5)_y\})$}
\begin{center}
\includegraphics[width=14cm]{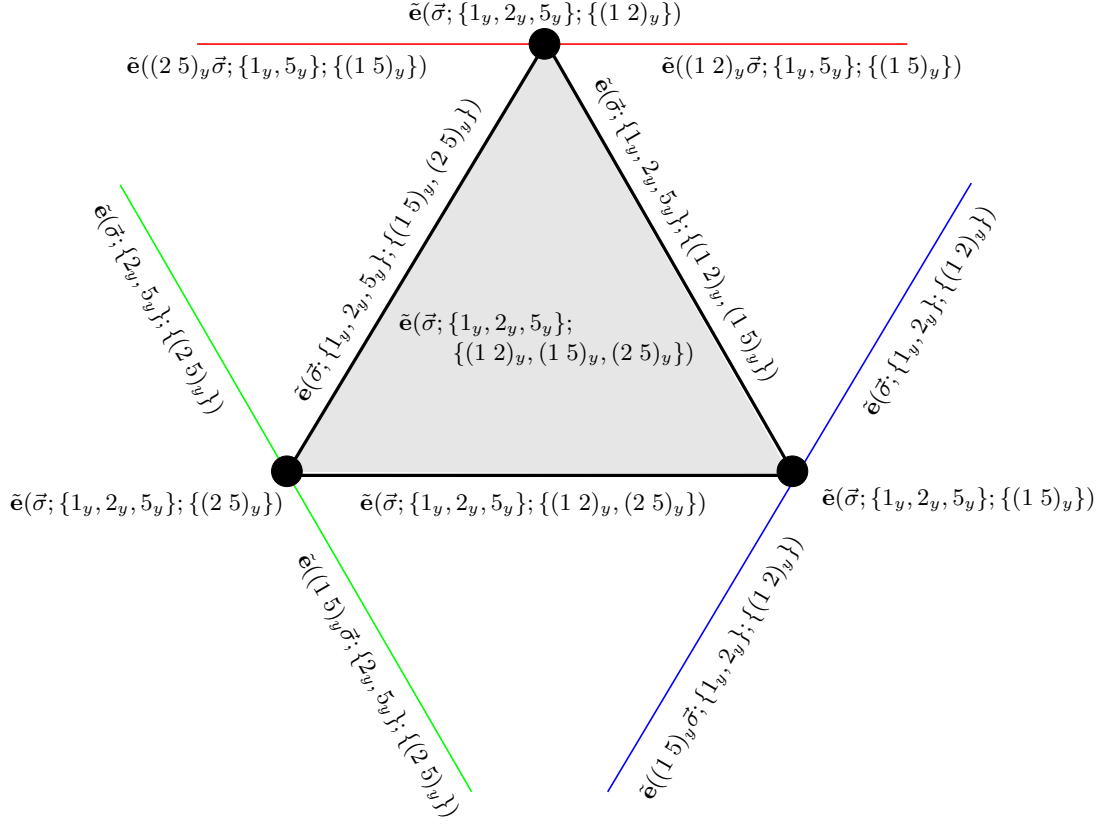}
\end{center}
\caption{The cell stucture over a non-transverse triple intersection in the discriminant at, for example, $\e((25134_x, 41253_y, 35241_z); \{1,2,5\}_y)$.}
\label{F:cellstruct}
\end{figure}

By the Leibniz rule, the boundaries of cells in this complex decompose into an \emph{external} component inherited from the boundary maps in $\Cell_\bullet(S_m)$ and an \emph{internal} component induced by the combinatorics of the join of simplices. The external boundary component introduces new equalities of coordinates in $\bC$, whereas the internal component deletes transpositions from $\rho$.

That is, the boundary of $\e(\vsig; \bC; \rho)\in C_*(\tS_m)$ is given by

\begin{equation*}
d(\e(\vsig; \bC; \rho)) = \sum_{\bC'} \pm \e(\vsig; \bC'; \rho)+ \sum_{\rho_i \in \rho} \pm \e(\vsig; \bC; \rho \setminus \{\rho_i\}),
\end{equation*}

\noindent where $\bC'$ range over coarsenings of the partition $\bC$ produced by combining precisely two elements of $\bC$ so that the resulting sets are admissible for $\vsig$, and the signs alternate with respect to the lexicographic ordering in both sums.

This cell structure contains both the singularity data from the original discriminant and combinatorial data analagous to that in Vassiliev's auxillary spectral sequences from \cite{Vas90}. This wealth of combinatorial data allows us to perform detailed analysis at the chain level. Indeed, we will see that there is a canonical choice of chain representative for a plumbers' knot invariant in our blowup.  


\subsection{The complexity filtration}

In order to construct the unstable Vassiliev spectral sequence on the spaces of plumbers' knots, we require an increasing \emph{complexity} filtration on the space $S_m$ which will induce such a filtration on its cell complex. However, as all maps in a cell share the same singularity data, we can directly define the filtration on $\Cell_\bullet(S_m)$. This filtration will then lift to the cell complex for $\tS_m$. 

As we wish to compare this spectral sequence to the classical Vassiliev spectral sequence, we will construct our filtration so that ``stable" singularity data appears in the expected filtration. The filtration on other maps is then determined by choosing the greatest complexity amongst cells which such a cell bounds. 

Due to the rigid geometry of plumbers' maps, very few configurations of transverse intersections of pipes are possible. Isolated triple points only occur when three pipes, one parallel to each of the coordinate axes, intersect in a single point. It is impossible to produce an isolated quadruple (or higher) point. Since double and triple points are the only singularities considered in classical Vassiliev theory, thus suggests that plumbers' knots are naturally suited to this analysis. As we can use plumbers' knots as a model for classical knot theory, this observation gives further circumstantial evidence that Vassiliev's invariants should be a complete system of knot invariants.

\begin{defn}\label{D:simple} Call a plumbers' curve whose only singularities are transverse double points \emph{simple}. If $\e \in \Cell_\bullet(S_m)$ is a cell whose points are simple plumbers' curves, call $e$ \emph{simple} as well.
\end{defn}

Simple curves are those with generic singularity data. However, some pipes may intersect two or more other pipes, which is an unstable condition.

\begin{defn}\label{D:complexity} Let $\e \in \Cell_\bullet(S_m)$ be simple.  Define the \emph{complexity} of $\e$, \cx$(\e)$, to be the number of double points of a curve in $e$.

 For $\e \in \Cell_k(S_m)$ which contains a singularity other than isolated double points, define
\begin{equation*}
\textrm{\sc{cx}}(\e) = \max\{\textrm{\sc{cx}}(\mathbf{f}) | \mathbf{f} \in \Cell_{k+1}(S_m) \textrm{ and } \e \in \partial(\mathbf{f})\}.
\end{equation*}
\end{defn}

As triple points only occur in the boundary of cells with a pair of double points, we recover that isolated triple points have complexity 2. Similarly, pauses of three consecutive zero-length pipes are in the boundary of the transverse double point where the curve ``turns back through itself" in the span of four pipes.

\begin{defn}
Let $F_p\Cell_\bullet(S_m) = \{\e \in \Cell_\bullet(S_m) | \textrm{\sc{cx}}(\e) \leq p\}$. 
\end{defn}

We observe that the maximal number of transverse self intersections of a plumbers' curve occurs when all of its defining vertices lie in a single plane. The maximal number of transverse intersections rectilinear motion can produce in a fixed number of pipes tells us that $F_{{m-1 \choose 2}}\Cell_\bullet(S_m) = \Cell_\bullet(S_m)$.

Using the descent of this filtration to the space $S_m$, we may generalize the notion of isotopy to all plumbers' curves. 

\begin{defn}
Two singular plumbers' curves $\phi, \phi'\in S_m$ are \emph{isotopic} if there exists a path $\Phi\colon\thinspace I\to S_m$ with $\Phi(0)=\phi, \Phi(1)=\phi'$ and $\Phi(I) \subseteq (F_p \setminus F_{p-1}) S_m$ for some $p$.
\end{defn}

Finally, we lift this filtration to $F_pC_*(\tS_m) = \pi_\#^{-1}(F_p\Cell_\bullet(S_m))$. By construction, the boundary maps in $C_*(\tS_m)$ can never decrease complexity. As the suspension maps $S_m \into S_{m+1}$ do not change the image of a curve, they also respect this filtration. 

\subsection{The Vassiliev derivative of a plumbers' knot invariant} 

The standard approach to the study of Vassiliev invariants has been through the Vassiliev derivative, introduced by Birman and Lin \cite{BirmanLin} and popularized by Bar-Natan \cite{NatanOnVas}. The classical Vassiliev derivative is defined for $n$-fold double points and can be extended through the 4-term relation for triple points, but fails to see more degenerate singularities of smooth knots. Here, we define an analogue of the Vassiliev derivative for invariants of plumbers' curves across any choice of singular cell. 

\begin{defn} \label{D:cobound}
Fix $\te(\vsig; \bC; \rho) \in C_{*}(\tS_m)$, $C_i \in \bC$. Let $\rho(C_i)$ denote those transpositions of $\rho$ supported on $C_i$. If $\rho(C_i)$ is not sequential for any $\vsig$, define the $C_i$ coboundary of $\te(\vsig; \bC; \rho)$  to be zero. 

If $\rho(C_i)$ is sequential for some $\vsig'$, it is sequential for exactly two such, both with the property that $\te(\vsig';\bC;\rho) = \te(\vsig;\bC;\rho)$. These two choices of $\vsig'$ will differ by reversal of the order in which the elements of $C_i$ appear in the permutation. Let $\vsig[\rho(C_i)]^+$ be the one which occurs first in the underlying lexicographic ordering and $\vsig[\rho(C_i)]^-$ the other. Define the $C_i$-coboundary of $\te(\vsig; \bC; \rho)$ to be
\begin{equation*}
\delta_{C_i}(\te(\vsig; \bC; \rho)) =  \te(\vsig[\rho(C_i)]^+; \bC;\rho \setminus \rho(C_i)) - (-1)^{|C_i|} \te(\vsig[\rho(C_i)]^-;\bC;\rho \setminus \rho(C_i)) 
\end{equation*}
\end{defn}

The fact that the $C_i$ are disjoint immediately implies

\begin{lem}\label{L:partialscommute}
Let $\te(\vsig; \bC;\rho)\in C_*(\tS_m)$. For any pair $C_1, C_2 \in \bC$, $\delta_{C_1}\delta_{C_2} \te(\vsig; \bC;\rho) = \delta_{C_2}\delta_{C_1} \te(\vsig; \bC;\rho).$
\end{lem}

\begin{defn}
Define the \emph{total coboundary} of $\te=\te(\vsig; \bC; \rho)$, written $\delta_\bC(\te)$, to be the element of $\Cell_{3n-3}(K_m)$ resulting from composing, in any order, all of the $\delta_{C_i}$ for $C_i\in \bC$. 
\end{defn}

\refL{partialscommute} then says that the total coboundary is well defined.

Note that when $\bC$ has a single component $\delta_\bC(\te)$ is the ``signed difference" of two cells in $\Cell_{3n-3}(K_m)$, as in \refF{cellstruct}. In this sense, certain codimension one cells in the blowup ``separate" pairs of cells from $P_m$ containing plumbers' curves. With this intuition, the following definition agrees with our understanding from Section \ref{conceptual} of the Vassiliev derivative.

\begin{defn}\label{D:VasDeriv}
Let $[\alpha]\in \bar{H}^0(K_m)$ and $\te=\te(\vsig; \bC; \rho)\in C_{3m-4}(\tS_m)$. The \emph{Vassiliev derivative of $[\alpha]$ at $\te$}, $d_{\te}([\alpha])$, is $[\alpha](\delta_{\bC}(\te))$.
\end{defn}

In contrast to the classical Vassiliev derivative, this definition works for any singularity of plumbers' curve. A straightforward computation shows that

\begin{prop}
The Vassiliev derivative is an isotopy invariant for singular plumbers' curves.
\end{prop}

Our definition agrees with the classical definition when the plumbers' curves are sufficiently articulated and the singularities are of the expected variety.

\begin{defn}
Call an isolated singularity of a plumbers' curve \emph{stable} if it is separated from any other singular point by at least one vertex. A cell consisting of singular plumbers' curves whose singularities consist only of stable double and triple points is a \emph{stable cell}.
\end{defn}

We remark that our notion of stability differs from that which arises in the construction of the classical Vassiliev spectral sequence. As discussed in Section \ref{conceptual}, the classical approach makes use of a stable range in the combinatorics of the spectral sequence. Because the plumbers' curves form an inverse system of model spaces, we instead draw on geometric properties of maps which become fixed after the application of some number of stabilization maps.

Recall that a singular curve is said to respect a chord diagram if the endpoints of each chord are identified in the image of the map (c.f. \cite{NatanOnVas}). Chord diagrams are usually considered up to diffeomorphisms of the spine which do not change the order of the endpoints of the chords, and we will abuse notation and call such a class of chord diagrams a chord diagram. When we say a map respects a chord diagram, we will mean that it respects some member of its equivalence class. 

Each stable cell $\e$ has associated to it some maximal chord diagram which its elements respect, so to evaluate the Vassiliev derivative of $[\alpha] \in \mathcal{FT}_n$ across its lift $\te$, one evaluates a representative weight system for $[\alpha]$ on this chord diagram.

The following lemma justifies the term ``stable" and follows immediately from \refD{VasDeriv}. In particular, it says that, on stable cells, our notion of Vassiliev derivative agrees with that of Birman and Lin \cite{BirmanLin}.

\begin{lem} \label{L:stablecell}
Let $\e\in \Cell_\bullet(S_m)$ be a stable cell of complexity $n$ and $[\alpha]\in H^0(K_m)$. The codimension one lift of such a cell, $\te = \pi_{\#}^{-1}(\e) = \e \times \Delta^{n-1}$. Further, $d_{\te}([\alpha])$ is given by evaluation of a representative weight system for $[\alpha]$ on the maximal chord diagram respected by $\e$ or, equivalently, by evaluation of $[\alpha]$ on an alternating sum of plumbers' knots produced by resolving in all possible ways the singularities of some map in $\e$.
\end{lem}

\begin{example}
Let $\e\in \Cell_{3m-3-n}(S_m)$ be a stable cell whose points are singular curves with precisely $n$ double points. The cell $\e$ has singularity data $\bC = \{\{a_i, b_i\}_{d_i}\}$, $i \in \{1\dots n\}$ and lifts to a codimension one cell $\te(\vsig; \bC; \rho) = \e \times \Delta^{n-1}$ whose second factor has vertices are labeled by transpositions $\rho_i=\rho(\{a_i, b_i\}_{d_i})=(a_i \; b_i)_{d_i}$. 

One computes that $\delta_\bC(\te(\vsig; \bC; \rho)) = \sum_i (-1)^i \te(\vsig[\rho_1]^{\epsilon_{i,1}}[\rho_2]^{\epsilon_{i,2}}\cdots[\rho_n]^{\epsilon_{i,n}}; \bC; \emptyset)$, where $\epsilon_{i,j}$ is the $j$th digit of the binary representation of $i$ using digits from $\{+, -\}$.

Therefore, for $[\alpha]\in H^0(K_m)$, $d_{\te}([\alpha])$ is the alternating sum of the value of $[\alpha]$ on the $2^n$ cells of $K_m$ cobounding $\e$. 
\end{example}

The ability to define the Vassiliev derivative for any singularity of plumbers' curves along with the cell structure on $\tS_n$ provides us with a great deal of information. In particular, it provides us with a canonical chain representative for the dual to a plumber's knot invariant, which we now construct.

The condition in \refD{cobound} that  each $\rho(C_i)$ be sequential for $\vsig$ implies that non-zero Vassiliev derivatives only occur for cells of dimension $3m-4$. The internal faces of such a cell are indexed by forgetting one transposition $\tau$ in some $\rho_i$. Note that in order to have internal faces, $|\rho|$ must be greater than one.

Given such a face $\tf = \tf(\vsig; \bC; \rho \setminus \{\tau\})$, the collection of cells which are incident to the face are of two types: internal cofaces which also lie in $\pi_\#^{-1}(\e(\vsig; \bC))$ and external cofaces which appear in some $\pi_\#^{-1}(\e(\vsig'; \bC; \rho \setminus\{\tau\}))$.

The internal cofaces which are incident to $\tf$ are of the form $\te(\vsig; \bC; (\rho \setminus\{\tau\}) \cup \{\tau'\}\})$, for some $\tau'\in \Sigma_{C_i}$. Precisely those $\tau'$ whose addition to $\rho \setminus\{\tau\}$ results in a new collection sequential for $\vsig$ correspond to cells with non-zero Vassiliev derivative. 

Write $\rho_i = \{(\rho_i(1)\; \rho_i(2)), (\rho_i(2)\; \rho_i(3)), \dots,(\rho_i({k-1})\; \rho_i(k))\}$. There are two possibilities: $\tau$ is an ``endpoint", either $(\rho_i(1)\; \rho_i(2))$ or $(\rho_i({k-1})\; \rho_i(k))$, or removing $\tau = (\rho_i(\ell)\; \rho_i({\ell+1}))$ splits $\rho_i$ into two disjoint collections sequential for $\vsig$. 

In the first case, there are two choices of transposition $\tau'\in \Sigma_{C_i}$ whose addition will result in a collection sequential for $\vsig$: $\tau$ and $(\rho_i(1)\; \rho_i(k))$. Otherwise, any of the transpositions $(\rho_i(1)\; \rho_i({\ell+1}))$, $(\rho_i(\ell)\; \rho_i({\ell+1)})$, $(\rho_i(1)\; \rho_i({k}))$ or $(\rho_i(\ell)\; \rho_i({k}))$ ``reattach" them into a single collection sequential for $\vsig$, while the rest result in non-sequential collections. Thus, there are always either two or four internal cofaces which can contribute a non-zero coefficient to $\tf$.

In constrast, there are always two external cofaces incident to a given $\tf$ which can contribute non-zero coefficients. As mentioned above, these are the cells $\te(\vsig;\bC'; \rho \setminus\{\tau\}) \in \pi_\#^{-1}(\e(\vsig; \bC'))$ and $\te(\tau\vsig;\bC'; \rho \setminus \{\tau\}) \in \pi_\#^{-1}(\e(\tau\vsig; \bC'))$, where $\bC'$ is the refinement obtained from $\bC$ by ``splitting the appropriate $C_i$ along $\tau$" and $\tau$ acts on $\vsig$ by block permutation of the elements in these two new partition elements in $\bC$. 

From this information, can now deduce the following ``Taylor's theorem". 

\begin{thm}\label{T:chainreps}
Let $[\alpha]\in \bar{H}^0(K_m)$. The lift of its Alexander dual cycle $[\alpha^\vee]$ to $H_{3m-4}(\tS_m)$ has a chain representative given by $\tilde{\alpha}^\vee = \sum_{\te\in C_{3m-4}(\tS_m)} (-1)^{o(\te)}d_{\te}([\alpha])\te$. 
\end{thm}

\begin{proof}
We must check two things: that $\tilde{\alpha}^\vee$ agrees with a chain representative of $[\alpha^\vee]$ on cells which lift to homeomorphic copies of themselves, and that it is a cycle. 

The cells in $\Cell_{3m-4}(S_m)$ are all of the form $\e(\vsig;\{a_i, b_i\})$. These cells lift to cells of the form $\te=\te(\vsig; \{a_i, b_i\}; \{(a\; b)_i\})$. We see that $d_{\te}([\alpha])=(-1)^{o(\te)}([\alpha](\e(\vsig) - \e((a\;b)_i\vsig)$, the difference of the value of the invariant on the cobounding cells of $\e$, which is precisely the coefficient assigned to the cell by Alexander duality. Thus, if $\tilde{\alpha}^\vee$ is a cycle, it is a chain representative of the lift.

Now, it remains to show that for each cell $\tf\in C_{3m-5}(\tS_m)$, the total contribution of cells incident to $\f$ under the boundary map $d$ is zero. To do so, we will consider an arbitrary cell $\te\in C_{3m-4}(\tS_m)$ for which Vassiliev derivatve of $[\alpha]$ is non-zero, select one of its internal faces and compute the sum of the incidence coefficients of each of the face's cobounding cells. It suffices to consider internal faces of such cells, as every cell with a coface whose Vassiliev derivative is non-zero arises as such an internal face.

Write ${C_i}(\tau) = C_i' \sqcup C_i''$ for the refinement of $C_i$ by splitting at $\tau$, $\bC(\tau)$ for the corresponding refinement to $\bC$ and $\rho[\tau, i, j] = (\rho \setminus {\tau}) \cup (\rho_i\;\rho_j)$. Let $\partial$ be the standard coboundary of a cell in $C_*(\tS_m)$ and use our analysis of cofaces to compute that (up to a sign depending on choice of $\vsig$),

\begin{align*}
\partial(\tf(\vsig;\bC;\rho\setminus\{\tau\})) &= - \te(\vsig;\bC;\rho) + (-1)^{\ell}\te(\vsig;\bC;\rho[\tau, 1, \ell+1])+ (-1)^{k-\ell}\te(\vsig;\bC;\rho[\tau, \ell, k])\\
& \indent - (-1)^{k-1} \te(\vsig;\bC;\rho[\tau, 1, k])+\te(\vsig;\bC(\tau);\rho\setminus\{\tau\}) - \te(\tau\vsig;\bC(\tau);\rho\setminus\{\tau\})\\
&\indent + \textrm{cells with zero Vassiliev derivative.}
\end{align*}

Using \refL{partialscommute}, we can rewrite $\delta_{\bC}$ as $\delta_{\bC\setminus C_i}\delta_{C_i}$ and $\delta_{\bC\setminus C_i}\delta_{C_i'}\delta_{C_i''}$ for these two different types of cells. 

Expanding these coboundaries, we compute

\begin{align*}
d_{\delta(\tf)}([\alpha]) &= [\alpha] \bigl(
\delta_{\bC}\te(\vsig;\bC;\rho) + 
\delta_{\bC}\te(\vsig;\bC;\rho[\tau, 1, k]) + 
\delta_{\bC}\te(\vsig;\bC;\rho[\tau,1,\ell+1])\\
& \indent + \delta_{\bC}\te(\vsig;\bC;\rho[\tau,\ell, k])+
\delta_{\bC(\tau)}\te(\vsig;\bC(\tau);\rho\setminus\{\tau\}) + 
\delta_{\bC(\tau)}\te(\tau\vsig;\bC(\tau);\rho\setminus\{\tau\})\bigr)\\
&= [\alpha] \bigl( \delta_{\hat{\bC}_i} ( 
\delta_{C_i}(\te(\vsig;\bC;\rho) + 
\te(\vsig;\bC;\rho[\tau,1,k]) \\
& \indent + \te(\vsig;\bC;\rho[\tau,1,\ell+1]) + 
\te(\vsig;\bC;\rho[\tau,\ell,k])) \\
& \indent +\delta_{C_i'}\delta_{C_i''}(\te(\vsig;\bC(\tau);\rho\setminus\{\tau\}) + 
\te(\tau\vsig;\bC(\tau);\rho\setminus\{\tau\})))\bigr)\\
&= [\alpha] \bigl( \delta_{\hat{\bC}_i} ( 
\te(\vsig[\rho(C_i)]^-;\bC;\rho\setminus\rho(C_i)) - (-1)^{|C_i|}\te(\vsig[\rho(C_i)]^+;\bC;\rho\setminus\rho(C_i))\\
&\indent + \te(\vsig;\bC;\rho\setminus\rho(C_i)) - (-1)^{|C_i|}\te(w_0(C_i)\vsig;\bC\rho\setminus\rho(C_i)\bigr).
\end{align*}

All sixteen resulting terms cancel, so $d_{\delta{\tf}}([\alpha]) = [\alpha](0) = 0$, as required. 

\end{proof}

Of course, our choice of $\tilde{\alpha}^\vee$ is only well defined up to a boundary in the chain complex. We make the canonical choice that this boundary contributes zero, and call this canonical representative the \emph{Vassiliev-Taylor series for $[\alpha]$}. For purposes of computation, one can identify representatives of $[\alpha^\vee]$ in $C_*(\tS_m)$ with fewer non-zero coefficients by choosing certain boundary contributions also to be non-zero. 

\begin{defn}
Let $[\alpha]\in H^0(\mathcal{K})$. Define the \emph{Vassiliev system} of $[\alpha]$ to be the collection $\{\tilde{\alpha}_5^\vee,\tilde{\alpha}_6^\vee,\dots\}$ of the Vassiliev-Taylor series of its restrictions $[\alpha_m]\in H^0(K_m)$.
\end{defn}

By \refT{chainreps}, every knot invariant (even those not of finite type) is completely determined by its Vassiliev system. 

\subsection{The unstable Vassiliev spectral sequence}

We can now construct analogues of the Vassiliev spectral sequence for the spaces of plumbers' knots. The cell complex on each plumbers' knot space will allow us to analyze these sequences explicitly. We will show that the inverse limit of these spectral sequences contains the collection of finite-type invariants.

\begin{defn}
Let $E^r_{p',q'}(m)$ be the homology spectral sequence of the complexity filtration on $\tS_m$ with $E^0$ page given by $E^0_{p',q'}(m) = \left(F_{p'}/ F_{p'-1}\right)C_{q'-p'}(\tS_m)$ and converging to $H_*(\tS_m)\cong H_*(S_m)$. The corresponding spectral sequence in cohomology $E_r^{p,q}(m)$ is obtained by reindexing $p = -p', q = (3m-4)-q'+2p'$, We call this the $m$th \emph{unstable Vassiliev spectral sequence}.
\end{defn}

By Alexander duality, $E_r^{p,q}(m) \implies \bar{H}^*(K_m)$. 

A stable cell $\te$ cannot lie in the boundary of any $\tilde{f} \in C_{3m-3}(\tS_m)$ because such a cobounding cell can only arise internally to $\pi_{\#}^{-1}(\e)$, and $\te$ is the cell of highest dimension lying over $\e$. Thus, any cycle involving $\te$ will represent non-zero homology class. Indeed, in this fashion we can identify a collection of non-zero cycles in $E_1^{-n,n}(m)$.

\begin{defn}
Fix $\te\in C_{3m-4}(\tS_m)$ a stable cell of complexity $n$. Let $[N(\te)]\in E_1^{-n,n}(m)$ be the unique minimal cycle containing $\te$.
\end{defn}

Such a cycle $[N(\te)]$ is simply the sum of all of the cells which can be reached from $\te$ by passing through sequences of non-zero codimension one boundary cells of complexity $n$. It can be chosen to be minimal because the filtration quotients are, up to homotopy, wedges of spheres. Because $\e$ is stable, such a cell cannot be the boundary of any other cell, so this cycle is non-trivial.

Using such cycles, we can see that finite type invariants arise in the expected complexity in limit of the unstable spectral sequences.

\begin{thm}
$\mathcal{FT}_n \into \invlim E_\infty^{-n,n}(m)$.
\end{thm}

\begin{proof}
Let $[\alpha]\in \bar{H}^0(\mathcal{K})$ be an invariant of type $n$. Fix a representative linear combination of weight systems for $[\alpha]$, a linear combination of chord diagrams $\sum c_i$ with which this representative pairs non-trivially and a collection of singular curves $\Gamma=\{\gamma_i | \gamma_i \textrm{ respects } c_i\}$. 

Choose an integer $m(\Gamma)$ large enough that all of the $\gamma_i\in \Gamma$ are represented in $S_{m(\Gamma)}$ by stable curves lying in cells $\e_i$. This ensures that we can resolve the singularities individually, so all of the topological knot types necessary to apply \refL{stablecell} are represented in $K_{m(\Gamma)}$. Write $[\alpha_{m(\Gamma)}]\in H^0(K_{m(\Gamma)})$ for the restriction of $[\alpha]$ to $K_m$. Then $d_{\sum_i \te(\gamma_i)}([\alpha_{m(\Gamma)}]) = \sum_i\langle [\alpha], c_i\rangle \neq 0$ and so the class $\sum_i \langle \alpha, c_i\rangle[N(\te_i)]\in E_\infty^{-n,n}(m)$ is nontrivial. 

Applying the universal property of the inverse limit, we see that $[\alpha]$ maps non-trivially to $\lim E_\infty^{-n,n}$.
\end{proof}






\bibliographystyle{plain}
\bibliography{../central_bibliography}

\end{document}